\documentclass{birkmult}

\usepackage{url}
\usepackage{amsmath}
\usepackage{amssymb}
\usepackage{amscd}
\usepackage{manfnt}
\usepackage{enumerate}
\usepackage[all]{xy}
\usepackage{graphicx}
\usepackage{MnSymbol}

\usepackage{tikz}
\usetikzlibrary{intersections,patterns,matrix}
\usepackage{pgfplots}

\usepackage[colorinlistoftodos]{todonotes}

\numberwithin{equation}{section}

\newtheorem{theorem}{Theorem}[section]
\newtheorem*{theorem*}{Theorem}
\newtheorem{proposition}[theorem]{Proposition}
\newtheorem{lemma}[theorem]{Lemma}
\newtheorem{corollary}[theorem]{Corollary}

\theoremstyle{definition}
\newtheorem{definition}[theorem]{Definition}
\newtheorem{example}[theorem]{Example}

\newtheorem{convention}[theorem]{Convention}

\theoremstyle{remark}
\newtheorem{remark}[theorem]{Remark}

\newcommand{\mv}{\mathcal{V}}
\newcommand{\mvt}{\widetilde{\mathcal{V}}}
\newcommand{\mw}{\mathcal{W}}

\newcommand{\Hom}{\operatorname{Hom}}

\def\Z{\mathbb Z}
\def\R{\mathbb R}
\def\F{\mathbb F}
\def\C{\mathbb C}
\def\Q{\mathbb Q}

\def\O{\Omega}
\def\wt#1{\widetilde{#1}}

\def\cM{\mathcal{M}}
\def\cV{\mathcal{V}}
\def\cW{\mathcal{W}}

\def\cM{\mathcal{M}}
\def\ee{\mathfrak{e}}
\def\ff{\mathfrak{f}}
\def\bmodt{\textrm{mod}-2}

\def\ol#1{\overline{#1}}

\DeclareMathOperator\im{im}
\DeclareMathOperator\var{var}

\DeclareMathOperator\lk{lk}

\begin{document}
\title{Real Seifert forms, Hodge numbers and Blanchfield pairings}

\author{Maciej Borodzik}
\address{Institute of Mathematics, University of Warsaw, ul. Banacha 2,
02-097 Warsaw, Poland}
\email{mcboro@mimuw.edu.pl}

\author{Jakub Zarzycki}
\address{Institute of Mathematics, University of Warsaw, ul. Banacha 2,
02-097 Warsaw, Poland}
\email{jz371722@students.mimuw.edu.pl}


\begin{abstract}
In this survey article we present connections between Picard--Lefschetz invariants of isolated hypersurface singularities and Blanchfield forms for links. We emphasize the unifying role of Hermitian Variation Structures introduced by N\'emethi.
\end{abstract}

\maketitle

\section{Introduction}
Understanding a mathematical object via decomposing it into simple pieces is a very general procedure in mathematics, which can be seen in various branches and various fields.
These procedures, often very different from each other, sometimes share common properties. In some cases, one mathematical object is defined in several fields and one
procedure of decomposing is known under different names in different areas of mathematics.

The subject of this article is an object called a linking form over $\R[t,t^{-1}]$, which in singularity theory corresponds to a real Hermitian Variation Structure defined
by N\'emethi in his seminal paper \cite{Nem_real}. Classification of simple Hermitian Variation Structures is an instance of a procedure known in algebraic geometry and algebraic topology
as \emph{d\'evissage}, which  --- at least for linking forms over $\R[t,t^{-1}]$ --- is a refinement of a primary decomposition of a torsion module over a PID.

These two points of view on linking forms: the Hodge-theoretical one and the algebraic one, give possibility to apply methods of one field to answer
questions that arise in another field. In this way,
the first author and N\'emethi gave a proof of semicontinuity of a spectrum of a plane curve singularity \cite{BoNe_spec} using Murasugi inequality of signatures.
Conversely, the Hodge theoretic aspect of linking forms, allows us to quickly compute knot invariants based on a small piece of data: an exemplary calculation
is shown in Example~\ref{ex:8_20}.

Another feature of Hodge-theoretical perspective is the formula for the Tristram--Levine signature, which we state in Proposition~\ref{prop:jump}. This formula 
allows us to define the analog of the Tristram--Levine signature for twisted Blanchfield pairings, compare Definition~\ref{def:twisig}. Many existing constructions
of similar objects involve a choice of a matrix \emph{representing} a pairing, see \cite[Section 3.4]{BCP}. However, finding a matrix representing given pairing,
even for pairings over $\C[t,t^{-1}]$ is not a completely trivial task, see e.g. \cite[Proposition 3.12]{BCP}. The approach through Hodge numbers allows us to bypass
this difficulty.

The structure of the paper is the following. In Section~\ref{sec:milnor} we recall the basics of Picard--Lefschetz theory. This section serves as a motivation for
introducing abstract Hermitian Variation Structures in Section~\ref{sec:HVS}. Section~\ref{sec:HVS-link} recalls the construction of a Hermitian Variation Structure
for general links in $S^3$. We also clarify the results of Keef, that were not completely correctly referred to in \cite{Hodge_type}.
In Section~\ref{sec:blanchfield} we give a definition of Blanchfield forms. We show that there is a correspondence between real Blanchfield forms and real
Hermitian Variation Structures associated with the link. Moreover, the classification of the two objects is very similar.

In the last section~\ref{sec:twisted} we sketch the construction of twisted Blanchfield pairings and introduce Hodge numbers for such structures. We show how
to recover the signature function from such a pairing. An example is given by Casson--Gordon signatures.

We conclude by remarking that in the paper we will mostly work over the field of real numbers. The complex case presents surprising technical issues. They
are mostly due to the fact that there are no irreducible Laurent polynomials over $\C$ that are symmetric, ie., $p(t^{-1})=\ol{p}(t)$. The case of complex
numbers is mentioned in Subsection~\ref{sub:few_words}. We refer to \cite{BCP} for more systematic treatment of the complex case.

\emph{Acknowledgments.} 
The article is based on a talk of the first author at the N\'emethi60 conference in Budapest in May 2019. 
The first authors is greatly indepted to Andr\'as N\'emethi for years of fruitful collaboration and for his guidance since they first met in 2009. 
The authors would like to thank to Anthony Conway and Wojciech Politarczyk for their
comments on the draft version of the paper.
The first author is supported by the National Science Center grant 2016/22/E/ST1/00040.
\section{Milnor fibration and Picard-Lefschetz theory}\label{sec:milnor}

Let $f : (\C^{n+1}, 0) \to (\C, 0)$ be a polynomial map with $0 \in \C^{n+1}$ an isolated critical point. 
\begin{theorem}[Milnor's fibration theorem, see \cite{Milnor_sing}]
For $\varepsilon > 0$ sufficiently small, the map $\Psi : S_\varepsilon^{2n+1} \setminus f^{-1}(0) \to S^1$
given by $\Psi(z) = \frac{f(z)}{\|f(z)\|}$ is a locally trivial fibration. The fiber $\Psi ^{-1}(1)$ has the homotopy type of a wedge sum of some finite number of spheres $S^n$.
\end{theorem}
Let $F_t$ be the fiber $\Psi^{-1} (t)$.
The \emph{geometric monodromy} $h_t$ is a diffeomorphism $h_t : F_1 \to F_t $, smoothly depending on $t$, which  corresponding to the trivialization of the Milnor fibration
on the arc of $S^1$ from $1$ to $t$.
Note that $h_t$ is well-defined only up to homotopy.
\begin{definition}
The \emph{homological monodromy} is the map $h : H_n(F_1; \Z) \to H_n(F_1; \Z)$ induced by the monodromy.
\end{definition}
The homological monodromy is not the only invariant that can be associated with the Milnor fibration.
Take a cycle
$\alpha \in H_n(F_1, \partial F_1;\Z)$. The image $h_1(\alpha)$ has the same boundary as $\alpha$. Hence, $h_1(\alpha) - \alpha$ is an absolute cycle.
\begin{definition}
  The \emph{variation map} $\var\colon H_n(F_1,\partial F_1;\Z)\to H_n(F_1;\Z)$ is the map defined as $\var\alpha=h_1(\alpha)-\alpha$.
\end{definition} 
\begin{remark}
  Poincar\'e--Lefschetz duality for $F_1$ implies that $H_n(F_1,\partial F_1;\Z)\cong \Hom(H_n(F_1;\Z),\Z)$. Therefore, the variation map can
  be regarded as a map from $H_n(F_1;\Z)^*$ to $H_n(F_1;\Z)$.
\end{remark}
We can also define a bilinear form based on linking numbers of $n$-cycles in $S^{2n+1}$.
\begin{definition}
The \emph{Seifert form} is the map $L : H_n(F_1, \Z) \times H_n(F_1, \Z) \to \Z$ given by $L(\alpha, \beta) = lk(\alpha, h_{\frac{1}{2}} \beta)$.
\end{definition}
Here $\lk(A,B)$ is the generalized linking pairing of two disjoint $n$-cycles in $S^{2n+1}$. A classical definition is that we have $H_n(S^{2n+1}\setminus B;\Z)\cong \Z$.
Then, we define $\lk(A,B)$ as the class of $A$ in $H_n(S^{2n+1},\setminus B)$. 

There are relations between the variation map, the Seifert form and the monodromy. References include \cite[Lemma 4.20]{Zoladek} and \cite{AVG2}.
\begin{theorem}\label{thm:picard_lefschetz_package}
  The Seifert form, the variation map, the monodromy and the intersection form on $H_n(F_1;\Z)$ are related by the following formulae:
  \begin{align*}
    L(\var a,b)&=\langle a,b\rangle\\
    \langle a,b\rangle&=-L(a,b)+(-1)^{n+1} L(b,a)\\
    h&=(-1)^{n+1}\var(\var^{-1})^*.
  \end{align*}
  Here $\langle\cdot,\cdot\rangle$ denotes the intersection form on $H_n(F_1;\Z)$.
\end{theorem}
Theorem~\ref{thm:picard_lefschetz_package} is a motivation to introduce Hermitian Variation Structures, which are the subject of the next section.

\section{Hermitian Variation Structures and their classification}\label{sec:HVS}
\subsection{Abstract definition}
Let $\F$ be a field of characteristic zero. By $\ol{\cdot}$ we denote the involution of $\F$: if $\F=\C$, then it is a complex conjugation, 
if $\F=\R,\Q$, then the involution is the identity. Set $\zeta=\pm 1$.
\begin{definition}\label{def:hvs}
A $\zeta$-\emph{Hermitian variation structure} over $\F$ is a quadruple $(U; b, h, V)$ where 
\begin{itemize}
  \item[(HVS1)] $U$ is a finite dimensional vector space over $\F$;
  \item[(HVS2)] $b : U \to U^*$ is a $\F$-linear endomorphism with $\overline{b^* \circ \theta} = \zeta b$, where $\theta : U \to U^{**}$ is a natural isomorphism;
  \item[(HVS3)] $h : U \to U$ is $b$-orthogonal, that is $\overline{h}^* \circ b \circ h = b$;
  \item[(HVS4)] $V: U^* \to U$ is a  $\F$-linear endomorphism with $\overline{\theta^{-1} \circ V^*} = - \zeta V \circ \overline{h^*}$ and $V \circ b = h - I$.
\end{itemize}
\end{definition}
The motivation is clearly Picard--Lefschetz theory. Suppose $f\colon(\C^{n+1},0)\to(\C,0)$ is a polynomial map as in Section~\ref{sec:milnor}.
The following result is a direct consequence of Theorem~\ref{thm:picard_lefschetz_package}.
\begin{proposition}
  Consider the quadruple $(U,b,h,V)$, where $U=H_n(F_1;\C)$, $b\colon H_n(F_1;\C)\to H_n(F_1,\partial F_1;\C)=\Hom_{\C}(H_n(F_1;\C);\C)$ is the Poincar\'e--Lefschetz duality,
  $h\colon U\to U$ is the homological monodromy and $V$ is the variation map. Then $(U,b,h,V)$ is a Hermitian Variation Structure over $\C$ with $\zeta=(-1)^n$.
\end{proposition}
Relations (HVS3) and (HVS4) suggest that having two of the three operators $b,h$ and $V$ we can recover the third one. This is true under some conditions,
which we are now going to spell out.

\begin{lemma}\ \label{lem:determined}
  \begin{itemize}
    \item[(a)] If $b$ is an isomorphism then  $V = (h-I)b^{-1}$. The HVS is determined by the triple $(U; h, b)$
    \item[(b)] If V is an isomorphism then $h = - \zeta V \overline{(\theta^{-1} \circ V^*)^{-1}}$ and $b = -V^{-1} - \zeta\overline{(\theta ^{-1} \circ V^*)^{-1}}$. So $V$ determines the HVS.
\end{itemize}
\end{lemma}
\begin{definition}
  The HVS such that $b$ is an isomorphism is called \emph{nondegenerate}. If $V$ is an isomorphism, we say that the HVS is \emph{simple}.
\end{definition}

\subsection{Classification of HVS over $\C$}\label{sub:class}
In \cite{Nem_real} N\'emethi provides a classification of simple HVS over $\F=\C$. This classification is based on a Jordan block decomposition
of the operator $h$. Note that we do not usually assume that all the eigenvalues of the monodromy operator are roots of unity,
as is the case of HVS associated with isolated hypersurface singularities.

As in \cite{Nem_real} we first list examples of HVS and then we state the classification result.
In the following we let $J_k$ denote the $k$-dimensional matrix 
$\{c_{ij}\}$, with $c_{ij}=1$ for $j=i,i+1$ and $c_{ij}=0$ otherwise, that is, $J_k$ is the single Jordan block of size $k$.

\begin{example}
Let $\nu\in\mathbb{C}^*\setminus S^1$ and $\ell\ge 1$. Define
\[
\mv_{\nu}^{2\ell}=\left(\mathbb{C}^{2\ell};%
\left(\begin{matrix}0&I\\\zeta I&0\end{matrix}\right),%
\left(\begin{matrix}\nu J_\ell&0\\ 0&\frac{1}{\bar{\nu}}{J_\ell^*}^{-1}\end{matrix}\right),%
\left(\begin{matrix}0&\zeta(\nu J_\ell-I)\\\frac{1}{\bar{\nu}}{J_\ell^*}^{-1}-I&0\end{matrix}\right)\right).
\]
Then $\mv_{\nu}^{2\ell}$ is a HVS. Furthermore,
$\mv_{\nu}^{2\ell}$ and
$\mv_{1/\ol{\nu}}^{2\ell}$ are isomorphic.
\end{example}

Before we state the next example, we need a simple lemma.
\begin{lemma}\label{lem:BB}
  Let $k\ge 1$ and $\zeta=\pm 1$. Up to a real positive scaling, there are precisely two non-degenerate matrices $b^k_{\pm}$
such that
\[\ol{b^k_{\pm}}^*=\zeta b\textrm{ and }J_k^*b^k_{\pm}J_k=b^k_{\pm}.\]
\end{lemma}
The entries of $b^k_{\pm}$ satisfy $(b^k_{\pm})_{i,j}=0$ for $i+j\le k$ and $b_{i,k+1-i}=(-1)^{i+1}b_{1,k}$.
Moreover, $(b^k_{\pm})_{1,k}$ is a power of $i$.
\begin{convention}\label{conv:sign_convention_mv}
By convention, we choose signs in such a way that \newline $(b^k_{\pm})_{1,k}=\pm i^{-n^2-k+1}$, where $n$ is such that $\zeta=(-1)^n$.
\end{convention}

Using $b^k_{\pm}$ we can give an example of a HVS corresponding to the case $\mu\in S^1$.
\begin{lemma}\label{lem:simple_in_circle}
  Let $\mu\in S^1$ and $k\ge 1$ be an integer. Up to isomorphism, there are two non-degenerate HVS such that $h=\mu J_k$.
  These structures have $b=b^k_+$ and $b=b^k_-$, respectively.
\end{lemma}
For these two structures we use the notation:
\[\mv^k_{\mu}(\pm 1)=\left(\mathbb{C}^k;b^k_{\pm},\mu J_k,(\mu J_k-I)(b^k_\pm)^{-1}\right).\]
  These two structures are simple unless $\mu=1$.
  For $\mu=1$ we need another construction of a simple HVS.

\begin{lemma}\label{lem:mvtilde}
Suppose $k\ge 2$. There are two degenerate HVS with $h=J_k$. These are:
\[\mvt^k_1(\pm 1)=\left(\mathbb{C}^k;\widetilde{b}_\pm,J_k,\widetilde{V}_{\pm}^k\right),\]
where
\[\widetilde{b}^k_\pm=\left(\begin{matrix} 0 & 0\\ 0&b^{k-1}_\pm \end{matrix}\right)\]
and $\widetilde{V}_\pm^k$ is uniquely determined by $b$ and $h$.
Moreover,   $\mvt^k_1(\pm 1)$  is simple.
\end{lemma}

While Lemma~\ref{lem:mvtilde} deals with the case $k\ge 2$, there remains the case $k=1$.
Then, with $\mu=1$, that is, $h=1$, all possible structures can be enumerated explicitly.
These are the following.
\begin{align*}
\mv^1_1(\pm 1)&=(\mathbb{C},\pm i^{-n^2},I,0)\\
\mvt^1_1(\pm 1)&=(\mathbb{C},0,I, \pm i^{n^2+1})\\
\mathcal{T}&=(\mathbb{C},0,I,0).
\end{align*}
From all these examples the structures $\mv^k_1(\pm 1)$ and $\mathcal{T}$ are non--simple,
and $\mvt^1_1(\pm 1)$ are simple.

Concluding, for any $\mu\in S^1$ and in each dimension $k$, there are precisely two non-equivalent simple
variation structures with $h=\mu J_k$.
We use the following uniform notation for them:
\begin{equation}\label{eq:wk}
\mw^k_\mu(\pm 1)=%
\begin{cases}
\mv^k_\mu(\pm 1)&\text{if $\mu\neq 1$}\\
\mvt^k_1(\pm 1)&\text{if $\mu=1$.}
\end{cases}
\end{equation}
The following result is one of the main results of \cite{Nem_real}.
\begin{theorem}\label{thm:classification}
A simple HVS is uniquely expressible as a sum of indecomposable ones up to ordering of summands and up to
an isomorphism. The indecomposable pieces are
\begin{eqnarray*}
 & \mw^k_\mu(\pm 1)&\text{ for $k\ge 1$, $\mu\in S^1$}\\ 
& \mv^{2\ell}_\nu &\text{ for $\ell\ge 1$, $0<|\nu|<1$.}\end{eqnarray*}
\end{theorem}

\begin{definition}\label{def:hodge_numbers}
  Let $\mathcal{M}$ be a simple HVS. The \emph{Hodge number} $p^k_{\mu}(\pm 1)$ for $\mu\in S^1$ is the number of times the structure $\mw^k_\mu(\pm 1)$ enters
  $\mathcal{M}$ as a summand. The Hodge number $q^\ell_\nu$ for $|\nu|\in(0,1)$ is the number of times the structure $\mv^{2\ell}_\nu$ enters
  $\mathcal{M}$ as a summand.
\end{definition}

For an isolated hypersurface singularity, the whole `Picard--Lefschetz package', that is, the monodromy, the variation map, the intersection form and
the Seifert form, are defined over the integers. Passing to $\C$ in the definition of a Hermitian Variation Structure means that some information
is lost. While we do not know how to recover the part coming from integer coefficients, the part of data coming from real coefficients
is easily to see.

Suppose $\cM=(U,b,h,V)$ is a HVS over $\R$. We construct a complexification of $\cM$ by considering $\cM_\C=(U\otimes \C,b\otimes\C,h\otimes\C,V\otimes\C)$.
Using Definition~\ref{def:hodge_numbers} we can associate Hodge numbers with $\cM_\C$. 
The following result is implicit in \cite{Nem_real}, see also \cite[Lemma 2.14]{Hodge_type}.
\begin{lemma}\label{lem:symmetry}
  The Hodge numbers of $\cM$ satisfy 
  \[p_{\ol{\mu}}^k(u)=p_{\mu}^k((-1)^{k+1+s}\zeta u)\textrm{ and }q_{\ol{\nu}}^\ell=q_\nu^\ell.\]
  Here $s=1$ if $\mu=1$, otherwise $s=0$.
\end{lemma}

The definition of a HVS is a generalization of the definition of Milnor's isometric structure \cite{Milnor_iso}; compare also \cite{Neumann}. Lemma~\ref{lem:determined} implies
that if the intersection form is an isomorphism, then the HVS is determined by the underlying isometric structure. Classification Theorem~\ref{thm:classification}
shows, that the only simple degenerate HVS correspond to the eigenvalue $\mu=1$. This is the main feature of the concept of a HVS: it allows us to deal
with the case $\mu=1$.

\subsection{The $\bmodt$ spectrum}
The spectrum of an isolated hypersurface singularity was introduced by Steenbrink in \cite{Steenbrink}. It is an unordered $s$-tuple of rational numbers $a_1,\dots,a_s\in (0,n+1]$, where
$n$ is the dimension of the hypersurface and $s$ is the Milnor number. 
The spectrum is one of the deepest invariants of hypersurface singularities.
The definition of the spectrum
involves the study of mixed Hodge structures associated with a singular point. We now show, following
N\'emethi, that the $\bmodt$ reduction (the tuple $a_1\bmod 2,\dots,a_s\bmod 2$) of the spectrum can be recovered from Hodge numbers. In particular, for 
plane curve singularities, the whole spectrum is determined by the Hodge numbers.

\begin{theorem}
  Let $p^k_\mu(u)$ be the Hodge numbers of an isolated hypersurface singularity in $\C^{n+1}$. For any $\alpha\in(0,2)\setminus\{1\}$, the multiplicity
  of $\alpha$ in the $\bmodt$ spectrum is equal to
  \[\sum_{k=1}^\infty\sum_{\epsilon=\pm 1} k p_{\mu}^{2k}(\epsilon)+\sum_{k=1}^\infty\sum_{\epsilon=\pm 1} (k+1-\epsilon\lfloor\alpha\rfloor)p^{2k+1}_\mu(\epsilon),\]
  where $\mu=e^{2\pi i\alpha}$.
\end{theorem}

The integer part of the spectrum, ie. the case $\alpha\in\{1,2\}$ can be treated in a similar manner.

\section{HVS for knots and links}\label{sec:HVS-link}
From now on we assume that $\zeta=-1$, so we consider only $(-1)$-variation structures.
\subsection{Three results of Keef}
The monodromy, the variation and the intersection form for an isolated hypersurface singularity are defined homologically. The
construction does not involve any analytic structure, that is, we need only existence
of a topological fibration of the complement of the link of singularity over $S^1$.
Therefore, if we have any fibered link $L\subset S^3$, we can use the same approach as above to define 
a HVS for such link. With a choice of a basis of $H_1(F)$, where $F$ is the fiber,
the variation map is the inverse of the Seifert matrix.

The construction can be extended further: take a link with Seifert surface $S$ and associate to it a simple HVS with variation map $S^{-1}$.
Now the Seifert surface is defined only up to $S$-equivalence and need not be invertible in general.
We shall use results of Keef to show that every Seifert matrix is $S$-equivalent to a block sum of an invertible matrix and that this invertible matrix is well-defined up to
rational congruence (for an analogous result for knots refer to \cite[Theorem 12.2.9]{Kawauchi}. Therefore, a HVS for any link in $S^3$ is defined.

In this subsection we consider matrices over $\Q$. As shown in \cite{Tro}, not all the results carry over to the case of $\Z$.
\begin{proposition}[\expandafter{see \cite[Proposition 3.1]{Keef}}]\label{prop:keef1}
  Any Seifert matrix $S$ for a link $L$ is S-equivalent over $\Q$ to a matrix $S'$ which is a block sum of a zero matrix and an invertible matrix $S_{in}$.
\end{proposition}
\begin{proposition}[\expandafter{see \cite[Theorem 3.5]{Keef}}]\label{prop:keef2}
  Suppose $S=S_0+S_{in}$ and $T=T_0+T_{in}$ be two matrices over $\Q$, presented as block sums of a zero matrix (that is, $S_0$ and $T_0$)
  and an invertible matrix (that is, $S_{in}$ and $T_{in}$). The matrices $S$ and $T$ are $S$-equivalent if and only if they are
  congruent. Furthermore, if $S$ and $T$ have the same size, then congruence of $S$ and $T$ is equivalent to congruence of $S_{in}$ and $T_{in}$.
\end{proposition}
\begin{proposition}[\expandafter{see \cite[Theorem 3.6]{Keef}}]\label{prop:keef3}
  Two matrices $S$ and $T$ are $S$-equivalent if and only if their Seifert systems are isomorphic.
\end{proposition}
Here, a \emph{Seifert system} relative to a square matrix $S$ consists of the module $A_S=\Q[t,t^{-1}]/(tS-S^T)$ and a pairing on the torsion part of $A_S$
as defined in \cite[Section 2]{Keef}.

From these three results we deduce the following fact. This result was often used in \cite{Hodge_type}, but actually its proof was never written down in detail.
\begin{proposition}\label{prop:congruence}\label{prop:keef_top}
  Suppose $S$ is S-equivalent to matrices $S'$ and $S''$, which are both block sums of zero matrices $S'_0$ and $S''_0$ and $S'_{in}$, $S''_{in}$,
  such that $S'_{in},S''_{in}$ are non-degenerate. Then $S'_{in}$ and $S''_{in}$ are congruent.
\end{proposition}
\begin{proof}
  By Proposition~\ref{prop:keef2} it is enough to show that the sizes of $S'$ and $S''$ is the same. 
  As $\Q[t,t^{-1}]$ is a PID, the module $A_{S'}=A_{S''}$
  decomposes as a direct sum of the free part and the torsion part.
  The sizes of $S'_0$ and $S''_0$ are equal to the rank
  over $Q[t,t^{-1}]$ of the free part of the module.

  Let $TA$ denotes the torsion-part of $A_{S'}=A_{S''}$. The order of $TA$ is the degree of the polynomial $\det(tS'_{in}-{S'_{in}}^T)=\det(tS_{in}''-{S_{in}''}^T)$. As
  $S'_{in}$ and $S''_{in}$ are invertible, the degree of $\det(tS'_{in}-{S'_{in}}^T)$ is equal to the size of $S'_{in}$. Therefore, the sizes of $S'_{in}$ and $S''_{in}$
  are equal. By Proposition~\ref{prop:keef2}, this shows that $S'_{in}$ and $S''_{in}$ are S-equivalent.
\end{proof}

\begin{remark}
  One would be tempted to guess that given a matrix $S$, the size of $S_0$ is $\dim(\ker S\cap\ker S^T)$. Such remark was made in \cite[Section 2.2]{BoNe_spec} but it was nowhere used. In fact, it is
  false. For a counterexample, take
  \[S=\begin{pmatrix} 0 & 0 & 1 \\ 1 & 0 & 0 \\ 0 & 0 & 0\end{pmatrix}.\]
  One readily checks that $\ker S\cap\ker S^T=0$ but $S$ is S-equivalent to the matrix $(0)$. So $\dim S_0=1$.
\end{remark}
\begin{definition}
  Let $L\subset S^3$ be a link with Seifert matrix $S$. Suppose $S$ is S-equivalent to $S'$, which is a block sum of a zero matrix and an invertible matrix $S_{in}$.
  The Hermitian Variation Structure for $L$ is the Hermitian Variation Structure $\cM(L)$ for which the variation operator is the inverse of $S_{in}$.
\end{definition}
From Proposition~\ref{prop:keef_top} we deduce the following result.
\begin{corollary}
  The Hermitian Variation Structure $\cM(L)$ is independent on the S-equivalence class of the matrix $S$, ie. it is an invariant of $L$.
\end{corollary}

\subsection{HVS for links and classical invariants}
Given the link $L\subset S^3$ and the HVS $\cM(L)$ we define Hodge numbers for $L$. Denote them $p_\mu^k(\pm 1)$ and $q_\nu^\ell$.
The Hodge numbers determine the one-variable Alexander polynomial of $L$ over $\R$ and the signature
function.
To describe the relation in more detail, we introduce a family of polynomials.
\begin{align}
    B_1(t)&=(t-1),\ \ B_{-1}(t)=(t+1)\notag\\
    B_\mu(t)&=(t-\mu)(1-\ol{\mu}t^{-1}) & \mu\in S^1,\ \im\mu>0\label{eq:basic}\\
    B_\mu(t)&=(t-\mu)(1-\mu^{-1}t^{-1}) & \mu\in\R,\ 0<|\mu|<1\notag\\
    B_\mu(t)&=(t-\mu)(t-\ol{\mu})(1-\mu^{-1}t^{-1})(1-\ol{\mu}^{-1}t^{-1}) & \mu\notin S^1\cup\R,\ 0<|\mu|<1.\notag
\end{align}
The (Laurent) polynomials $B_\mu$ for $\mu\neq\{1,-1\}$ are characterized by the property that they have real coefficients, they are symmetric ($B_\mu(t)=B_\mu(t^{-1})$)
and they cannot be presented as products of real symmetric polynomials. 
Moreover, these are (up to multiplication by $t$) the characteristic polynomials of the monodromy operators associated with HVS $\cW^k_\mu$.
With notation~\eqref{eq:basic} we obtain (see \cite[Section 4.1]{BCP}):
\begin{proposition}\label{prop:alex}
  Let $L$ be a knot. Then the Alexander polynomial of $L$ is equal to
\begin{equation}\label{eq:deltaLB}
  \Delta_L(t)=\prod_{\substack{\mu\in S^1\\\im\mu\ge 0}}\prod_{\substack{k\ge 1\\ u=\pm 1}} B_{\mu}(t)^{p_\mu^k(u)}\cdot\prod_{\substack{0<|\nu|<1\\\im\nu\ge 0}}
  \prod_{\ell\ge 1}B_\nu(t)^{q^\ell_\nu}.
\end{equation}
\end{proposition}

Another result gives the minimal number of generators of the Alexander module of a knot $L$ over $\R[t,t^{-1}]$; see \cite[Section 4.3]{Hodge_type}.
\begin{proposition}\label{prop:nakanishi}
  Suppose $\Delta_L$ is not identically zero. The minimal number of generators of the Alexander module over $\R[t,t^{-1}]$ is equal to
  \[\max\left(\max_{\mu\in S^1} \sum_{k,u}p^k_\mu(u),\max_{0<|\nu|<1} \sum_{\ell} q^\ell_\nu\right).\]
\end{proposition}
The jumps of the Tristram--Levine signature function of a link can also be described in terms of Hodge numbers.
Before we state the result, recall that for a link $L$, the Tristram--Levine signature $\sigma_L(z)$ is the signature of the Hermitian matrix $(1-z)S+(1-\ol{z})S^T$,
where $S$ is the Seifert matrix for $L$.
The \emph{jump} of the signature function at a point $z_0$ is
\[j(z_0)=\frac12\left(\lim_{t\to 0^+} \sigma_L(e^{it}z)-\sigma_L(e^{-it}z)\right).\]
We will now show that the Hodge numbers determine signatures. We give the formula for $\sigma_{L}(z_0)$ in case there is no jump of the signature function, so
that we can avoid discussing average signatures.
For more general statements of this type we refer to \cite[Proposition 4.14]{Hodge_type}. 
Another source is \cite[Section 5]{BCP}.
Note that the formulae in \cite{BCP} have different shape, but they are
equivalent.
\begin{proposition}\label{prop:jump}
  Let $L$ be a link and $z_0=e^{ix}\in S^1$ ($x\in(0,\pi)$) be such that $z_0$ is not a zero of the Alexander polynomial of $L$. Then
  \begin{equation}\label{eq:sig_comp}
    \sigma_L(z_0)=
    -\sum_{y\in[0,x)}\sum_{\substack{u\in\{-1,1\}\\ k\textrm{ odd}}} up^k_{e^{iy}}(u)
    +\sum_{y\in(x,1)}\sum_{\substack{u\in\{-1,1\}\\ k\textrm{ odd}}} up^k_{e^{iy}}(u).
  \end{equation}
\end{proposition}

Propositions~\ref{prop:alex}, \ref{prop:jump} and \ref{prop:nakanishi} can be used to determine the Hodge numbers directly, without referring to explicit study
of the Jordan block decomposition.
\begin{example}\label{ex:8_20}
  Let $K=8_{20}$. From \cite{knotinfo} we read off that $\Delta_K=(t-\mu)^2(t-\overline{\mu})^2$ for $\mu = \frac{1}{2} (1 + i \sqrt{3})$.
  Moreover, the Nakanishi index (the minimal number of generators of the Alexander module of $K$) is $1$.

  From Proposition~\ref{prop:alex} we deduce that either $p_{\mu}^1(+1)+p_{\mu}^1(-1)=2$ or $p_{\mu}^2(\pm 1)=1$. Using
  Proposition~\ref{prop:nakanishi} we exclude the first possibility. We deduce that $p_{\mu}^2(u)=1$, $u\in\{-1,1\}$ and without extra data we
  cannot determine the sign $u$.

  We conclude from Proposition~\ref{prop:jump} that the Tristram--Levine signature of $K$ is zero except for $\mu$ and $\ol{\mu}$, where it attains
  the value $u$.

  This shows for example, that the maximal absolute value of the signature of $nK$ is $n$, so the knot $nK$ has unknotting number at least $n/2$, even though
  it is slice.
\end{example}
\subsection{Signatures, HVS and semicontinuity of the spectrum}
Hodge numbers can be used to provide the relation between the signature of the link of singularity and the $\bmodt$ spectrum.
For simplicity, we state the result for curve singularities in $\C^2$.
\begin{theorem}[see \expandafter{\cite[Corollary 4.15]{Hodge_type}}]\label{thm:spec_sig}
  Let $f\colon\C^2\to\C$ defines an isolated singularity with link $L$ and spectrum $Sp$. Suppose $x\in(0,1)$ does not belong to the spectrum and $1+x$ does
  not belong to the spectrum, either. Then
  \[\sigma_L(e^{2\pi i x})=-\# Sp\cap(x,x+1)+\# Sp\setminus[x,x+1].\]
\end{theorem}
Theorem~\ref{thm:spec_sig} can be regarded as a generalization of Litherland's formula expressing the signature of a torus knot
in terms of the number elements in $Sp_{p,q}\cap (x,x+1)$, where $Sp_{p,q}=\{\frac{i}{p}+\frac{j}{q},\ 1\le i<p,\ 1\le j<q\}$
is the spectrum of singularity $x^p-y^q=0$; see \cite{Litherland}.

Spectrum of singularity is semicontinuous under deformation of singularities. While stating the result of Steenbrink and Varchenko \cite{Steen2,Varchenko} is
beyond the scope of this survey, we note that in \cite{BoNe_spec}, Murasugi inequality for signatures of links was used to obtain semicontinuity results.

\section{Blanchfield forms}\label{sec:blanchfield}
We now pass to defining Blanchfield forms. In some sense, Blanchfield forms generalize Hermitian Variation Structures, although the connection
might be hard to observe at first. We restrict to the case of knots, referring to \cite{Hillman} for the case of links.
First, we need to set up some conventions. Suppose $R$ be a ring with involution
(usually we consider $R=\Z,\Q,\R$ with trivial involution or $R=\C$ with complex conjugation).
The ring $R[t,t^{-1}]$ has an involution given by $\ol{\sum a_jt^j}=\sum \ol{a_j}t^{-j}$.

\subsection{Definitions}
Let $K\subset S^3$ be a knot. Let $X=S^3\setminus K$. By Alexander duality  $H_1(X;\Z)=\Z$. Hurewicz theorem implies
the existence of a surjetion $\pi_1(X)\to\Z$. We call the cover of $X$ corresponding to this surjection is called the \emph{universal abelian cover} of $X$. 
We denote it by $\wt{X}$. The first homology group $H_1(\wt{X};\Z)$ has a structure of $\Z[t,t^{-1}]$-module, with multiplication by $t$ being induced
by the action of the deck transformation on $\wt{X}$. This module is called the \emph{Alexander module} of $K$. Usually it is denoted by $H_1(X;\Z[t,t^{-1}])$;
in Section~\ref{sec:blanchfield} we will denote it by $H$.

There is a sesquilinear pairing on $H$, which is a generalization of the notion of a linking form on a rational homology three-sphere.
\begin{theorem}[\cite{Blanchfield}]\label{thm:blanchfield}
  The linking pairing $H\times H\to \Q(t)/\Z[t,t^{-1}]$ is Hermitian and non-degenerate.
\end{theorem}
\begin{definition}
  The pairing of Theorem~\ref{thm:blanchfield} is called the \emph{Blanchfield pairing} of $K$.
\end{definition}
Besides of definiting the form, Blanchfield in \cite{Blanchfield} show how to calculate explicitly the Blanchfield form from the Seifert matrix.
\begin{theorem}\label{thm:form}
  Let $K$ be a knot and let $S$ be a Seifert matrix for $K$, assume the size of $S$ is $n$. Denote $\Lambda=\Z[t,t^{-1}]$. Then $H=\Lambda^n/(tS-S^T)\Lambda^n$
  and with this identification the Blanchfield pairing is $(x,y)\mapsto x^T(t-1)(S-tS^T)^{-1}\ol{y}\in \Q(t)/\Lambda$.
\end{theorem}
\begin{remark}
  There is some confusion in the literature about the correct statement of Theorem~\ref{thm:form}. We refer the reader to \cite{FP}, where various
  possibilities are discussed and some common mistakes are corrected.
\end{remark}
Theorem~\ref{thm:form} shows that a Seifert matrix of $K$ determines the Blanchfield pairing. The reverse implication is also true; see
e.g. \cite{Tro,Ran_mos}.
\begin{theorem}
  The S-equivalence class of a Seifert matrix of a knot $K$ is determined by the Blanchfield form.
\end{theorem}
The importance of a Blanchfield form in knot theory justifies the following abstract definition.
\begin{definition}
  Let $R$ be an integral domain with (possibly trivial) involution.
  Let $\Omega$ be the field of fractions of $R$

  A \emph{linking form} over $R$ is the pair $(M,\lambda)$, where $M$ is a torsion $R$-module and
  $\lambda\colon M\times M\to \Omega/R$ is a non-degenerate sesquilinear pairing. Here `non-degenerate' means
  that the map $M\to \ol{\Hom_{R}(M,\Omega/R)}$ induced by $\lambda$ is an isomorphism.
\end{definition}
We refer to Ranicki's books \cite{Ran_ex} and \cite{Ran_L} for a detailed study of abstract linking forms and their properties.

\subsection{Blanchfield pairing over $\R[t,t^{-1}]$}\label{sub:blar}
We will now study classification of Blanchfield pairings over $\R[t,t^{-1}]$. As in Subsection~\ref{sub:class} we will first give
some examples and then, based on these examples, we state the classification result. First we deal with the case $\mu\in S^1$.
\begin{definition}\label{def:forme}
  Let $\mu\in S^1$, $\im\mu>0$. Let $k>0$, $\epsilon\in\{-1,1\}$. The hermitian form $\ee(\mu,k,\epsilon)$ is a pair
  $(M,\lambda)$, where 
  \begin{align*}
    M&=\R[t,t^{-1}]/B_\lambda(t)^k\\
    \lambda(x,y)&=\frac{\epsilon x\ol{y}}{B_\mu(t)^k}.
  \end{align*}
\end{definition}
\noindent The second definition is for $\mu\notin S^1$.
\begin{definition}\label{def:formf}
  Suppose $\nu\in\C$, $\im\nu\ge 0$ and $0<|\nu|<1$. For $\ell>0$ we define the hermitian form $\ff(\nu,\ell)$ as a pair
  $(M,\lambda)$, where 
  \begin{align*}
    M&=\R[t,t^{-1}]/B_\lambda(t)^\ell\\
    \lambda(x,y)&=\frac{x\ol{y}}{B_\nu(t)^\ell}.
  \end{align*}
\end{definition}
Note that Definitions~\ref{def:forme} and~\ref{def:formf} do not cover the case $\mu=\pm 1$. These two cases are special, because $B_{\pm 1}(t)$ is not symmetric,
but they do not occur in knot case, because 
$\pm 1$ is never a root of the Alexander polynomial of a knot.

The following result goes back at least to Milnor, see \cite[Theorem 3.3]{Milnor_iso}. We present the statement from \cite{BF}, see also \cite{BCP}.
\begin{theorem}\label{thm:class}
  Suppose $(M,\lambda)$ is a non-degenerate linking form over $\R[t,t^{-1}]$ such that the multiplication by $(t\pm 1)$ is an isomorphism of $M$. Then
  $(M,\lambda)$ decomposes into a finite sum:
  \begin{equation}\label{eq:Mlambda}
    (M,\lambda)=\bigoplus_{i\in I} \ee(\mu_i,k_i,\epsilon_i)\oplus\bigoplus_{j\in J}\ff(\nu_j,\ell_j),
  \end{equation}
  where $\mu_i\in S^1$, $0<|\nu_j|<1$, and $\im\mu_i>0$, $\im\nu_j\ge 0$
  Such a decomposition is unique up to permuting factors.
\end{theorem}
Theorem~\ref{thm:class} motivates the following definition.
\begin{definition}
  Let $(M,\lambda)$ be as in the statement of Theorem~\ref{thm:class}. The number $e^k_\mu(\epsilon)$ (respectively $f^\ell_\nu$) is the number of times the form
  $\ee(\mu,k,\epsilon)$ (respectively $\ff(\nu,\ell)$) enters $(M,\lambda)$ as a direct summand.
\end{definition}

\subsection{Variation operators and linking forms}
Let $\cM$ be a simple HVS over $\R$ with variation operator $V$ with $\zeta=-1$. 
Let $S=V^{-1}$. Motivated by Theorem~\ref{thm:form} define the pairing $(M,\lambda)$ by
\begin{equation}\label{eq:formM}
M=\R[t,t^{-1}]^n/(tS-S^T)\R[t,t^{-1}]^n,\ \ \lambda(x,y)=x^T(t-1)(S-tS^T)^{-1}\ol{y}.
\end{equation}
We call this form the \emph{linking form associated to $\cM$}. We have the following result.
\begin{proposition}\label{prop:equal}
  Let $\mu\in S^1$, $\im\mu>0$. Suppose $\cM=\cV^k_\mu(\epsilon)\oplus\cV^k_{\ol{\mu}}((-1)^{k}\epsilon)$. Then,
  the linking form associated with $\cM$ is equal to $\ee(\mu,k,\epsilon)$.
\end{proposition}
\begin{proof}
  The statement is well-known to the experts. The underlying $\R[t,t^{-1}]$-modules are clearly isomorphic and the sign
  $\epsilon$ is determined by comparing appropriate signatures,
  see \cite{Kearton1,Kearton2} and also Conway's survey \cite[Section 4.2]{Conway_survey}.

  We think it is instructive to give an elementary proof
  of Proposition~\ref{prop:equal} in case $k=1$.
  The method of computing sign of a non-degenerate pairing over $\R[t,t^{-1}]/B_\mu(t)^k$ is as follows. Take an element $v\in\R[t,t^{-1}]/B_\mu(t)^k$
  and compute $\lambda(v,v)=q/B_\mu(t)^k$. If $q$ is coprime with $B_\mu$, then the sign of $q(\mu)$ (this is clearly
  a real number) is precisely the sign of $\ee(\mu,k,\epsilon)$. A proof of the last statement follows quickly from
  the proof of \cite[Proposition 4.2]{BF}.

  We will first compute the Seifert matrix $S$ and compute $\lambda(v,v)$ via~\eqref{eq:formM}.
  From Lemma~\ref{lem:BB} we have $b_{\epsilon}^1=-\epsilon i$. Therefore, the variation operator associated with $\cV^1_\mu(\epsilon)$ is
  $\epsilon i(\mu-1)$. The variation operator corresponding to $\cV^1_\mu(\epsilon)\oplus\cV^1_{\ol{\mu}}(-\epsilon)$ is thus
  equal to
  \[V=\epsilon \begin{pmatrix} i(\mu-1) & 0 \\ 0 & -i(\ol{\mu}-1)\end{pmatrix}.\]
  Hence
  \[S=V^{-1}=\frac{-i\epsilon}{|\mu-1|^2}\begin{pmatrix} \ol{\xi} & 0 \\ 0 &\xi\end{pmatrix},\]
  where $\xi=i(\mu-1)$. Write in polar coordinates $\xi=r\cos\phi+ir\sin\phi$. Then, $S$ is congruent to the matrix
  \[S=\frac{\epsilon}{r} \begin{pmatrix} \cos\phi & \sin\phi \\ -\sin\phi & \cos\phi \end{pmatrix},\]
  The module $\R[t,t^{-1}]/B_\mu(t)$ is isomorphic to the module $\R[t,t^{-1}]^2/(tS-S^T)\R[t,t^{-1}]^2$.

  Since $\det(S-tS^T)=tB_\mu(t)$, we have for any $v\in\R[t,t^{-1}]^2$:
  \[\lambda(v,v)=v^T(t-1)(S-tS^T)^{-1}v=v^T\frac{(t-1)\epsilon r}{t B_\mu(t)}\begin{pmatrix} (1-t)\cos\phi & -(1+t)\sin\phi\\ (1+t)\sin\phi & (1-t)\cos\phi\end{pmatrix}v\]
  Take now vector $v=(1,0)$ and consider its class in $\R[t,t^{-1}]^2/(tS-S^T)$, which we denote by the same letter. We obtain
  \[\lambda(v,v)=\frac{\epsilon(t-2+t^{-1})r\cos\phi}{B_\mu(t)}.\]
  Now the sign of $2-\mu-\ol{\mu}$ is positive. To see the sign of $\cos\phi$
  we note that $\im\mu>0$, hence $\mu-1$ is in the second quadrant, so $i(\mu-1)$
  is in the third one, thus $\cos\phi$ is negative.
\end{proof}
\begin{remark}
  An analog of Proposition~\ref{prop:equal} for $\mu\notin S^1$ is trivial, because the pairing is determined by the underlying module structure.
\end{remark}
The following result is an easy consequence of Proposition~\ref{prop:equal}.
\begin{theorem}\label{thm:compare}
  There is an equality $p^k_\mu(\epsilon)=e^k_\mu(\epsilon)$, $q^\ell_\nu=f^\ell_\nu$.
\end{theorem}

\section{Twisted Blanchfield forms and applications}\label{sec:twisted}
One of the features of the Hodge-theoretic point of view on Blanchfield pairings is that we can define signature-type invariants
of pairings on torsion $\R[t,t^{-1}]$-modules, which do not necessarily come from Seifert matrices.
In particular, we can easily define signature-type invariants for twisted Blanchfield pairings. 
This includes for instance so-called Casson-Gordon signatures.

\subsection{Construction of twisted pairings}
We begin with a general construction. For a 3-manifold $X$ we consider its universal cover $\wt{X}$. This space is acted upon by $\pi_1(X)$.
With $C_*(\wt{X})$ denoting the singular chain complex of $\wt{X}$, we can regard $C_*(\wt{X})$ as a left module over $\Z[\pi_1(X)]$.
Suppose that $M$ is a $(R,\Z[\pi_1(X)])$-module for some ring $R$ (by this we mean a left $R$-module and a right $\Z[\pi_1(X)$-module).
We define $C_*(X;M)=M\otimes_{\Z[\pi_1(X)]} C_*(\wt{X})$. This chain complex of left $R$-modules is called a \emph{twisted chain complex} of $X$.
Its homology is called the \emph{twisted homology} of $X$; see \cite[Section 6.1]{BCP}, \cite{Kirk-Livingston}.

A special instance of this operation is when we consider a representation $\beta\colon\pi_1(X)\to GL_d(R)$ for some ring $R$ with involution and integer $d>0$.
The space $R^d$ has a structure of right $\Z[\pi_1(X)]$-module: an action of $\gamma\in\pi_1(X)$ is the multiplication
the vector in $R^d$ by $\beta(\gamma)$ from the right. Taking $M=R^d$ we obtain the twisted chain complex $C_*(X;R^d_\beta)$ (we write the subscript $\beta$)
to stress that this is a twisted chain complex).

Let us specify our situation more and suppose $R=\F[t,t^{-1}]$ for some field $\F$.
Assume moreover that $\beta\colon\pi_1(X)\to GL_d(R)$ is a \emph{unitary} representation.
We have the following result.
\begin{proposition}[see \cite{BCP,MP,Powell}]\label{prop:twistedpairing}
  Suppose $\beta\colon \pi_1(X)\to GL_d(\F[t,t^{-1}])$ is such that $H_1(X;\F[t,t^{-1}]^d_\beta)$ is $\F[t,t^{-1}]$-torsion.
  There is a hermitian non-degenerate pairing:
  \[H_1(X;\F[t,t^{-1}]^d_\beta)\times H_1(X;\F[t,t^{-1}]^d_\beta)\to \F(t)/\F[t,t^{-1}].\]
\end{proposition}
\begin{definition}
  The pairing defined in Proposition~\ref{prop:twistedpairing} is called the \emph{twisted Blanchfield pairing}.
\end{definition}

\subsection{Twisted Hodge numbers and twisted signatures}
We specify now to the situation, when $\F=\R$ and $X=M(K)$,
the zero-framed surgery on a knot $K$ and let $\beta\colon\pi_1(X)\to GL_d(\R[t,t^{-1}])$ be a unitary representation such
that $H_1(X;\R[t,t^{-1}]_\beta^d)$ is $\R[t,t^{-1}]$-torsion. Assume furthermore that $H_1(X;\R[t,t^{-1}]_\beta^d)$ has no $(t\pm 1)$-torsion.
Then the pairing twisted Blanchfield pairing is defined and by Theorem~\ref{thm:class} above, it decomposes as a sum of $\ee(\mu,k,\epsilon)$ and $\ff(\nu,\ell)$.

\begin{definition}\label{def:tHn}
  The \emph{twisted Hodge number} $p_\mu^k(\epsilon)_\beta$ and $f_{\nu,\beta}^\ell$ is the number of times the summand
  $\ee(\mu,k,\epsilon)$, respectively $\ff(\nu,\ell)$ enters the decomposition~\eqref{eq:Mlambda}.
\end{definition}

Having defined twisted Hodge numbers, we can define twisted signatures via an analog of \eqref{eq:sig_comp}.

\begin{definition}\label{def:twisig}
  Suppose $\mu=e^{2\pi ix}$, $x\in(0,1/2)$.
  The function
  \[\mu\mapsto\sigma_\beta(\mu)=\sum_{\substack{k\textrm{ odd}\\ \epsilon=\pm 1}} \left(p^k_\mu(\epsilon)_\beta+2\sum_{y\in(0,x)} p^k_{e^{2\pi iy}}(\epsilon)\right)_\beta\]
  is called the \emph{twisted signature function}. The function is extended via $\sigma_\beta(\ol{\mu})=\sigma_\beta(\mu)$.
\end{definition}

There is a subtle difference between Definition~\ref{def:twisig} and Proposition~\ref{prop:jump}.
The classical result, Proposition~\ref{prop:jump} sums contribution the Hodge numbers in a range including $0$. Therefore it is perfectly possible that the signature function is equal to $1$ for all values close to $1$.
This is the case for example for the Hopf link.

Definition~\ref{def:twisig} sums over $y$ in an open interval $(0,x)$, so the previous behavior is impossible. This is not merely a technical issue: it seems difficult to extend the definition of twisted
signature to get a meaningful contribution of $\mu=1$. A possible explanation is the parity of $k$
in \cite[Lemma 2.20]{BCP}.

\subsection{A few words on case $\F=\C$}\label{sub:few_words}

The construction of Hodge numbers via classification of linking pairings can be done over $\C[t,t^{-1}]$. We can define $\ee(\mu,k,\epsilon)$ for $\mu\in S^1$,
and $\ff(\mu,k)$ for $0<|\mu|<1$.
The underlying module structure is $\C[t,t^{-1}]/(t-\mu)^k$.
However, the specific construction of the first case seems to be harder than in case over $\R$; see \cite[Section 2]{BCP}. Once this
technical difficulty is overcome, we can define twisted Hodge numbers and twisted signatures essentially via Definitions~\ref{def:tHn} and~\ref{def:twisig}.

An important instance of twisted signatures over $\C[t,t^{-1}]$ are signatures defined from
Casson--Gordon invariants introduced by Casson and Gordon, see \cite{CassonGordon1,CassonGordon2}.
In short, let $K$ be a knot and let $n$ be an integer. Consider the $n$-fold cyclic branched cover $L_n(K)$. Let $m$ be a prime power coprime with $n$.
For any non-trivial homomorphism $\chi\colon H_1(L_n(K);\Z)\to\Z_m$ we can construct a unitary representation $\pi_1(M(K))\to GL_n(\C[t,t^{-1}])$.
The signature associated to this representation via Definition~\ref{def:twisig} is called a \emph{Casson-Gordon signature} $\sigma_{\chi,m}\colon S^1\to\Z$.
Casson--Gordon sliceness obstruction can be translated into vanishing of some  Casson--Gordon signatures. The following result
is stated in \cite[Theorem 8.8, Corollary 8.16]{BCP} as a corollary of a result of Miller and Powell \cite{MP}.
\begin{theorem}\label{thm:cass_sig}
  Let $K$ be a slice knot. Then for any prime power $n$, there exists a metabolizer $P$ of the linking form on $H_1(L_n(K);\Z)$ such that for any prime power $q^a$
  and any non-trivial homomorphism $\chi\colon H_1(L_n(K);\Z_{q^a})$ vanishing on $P$, there is $b\ge a$ such that $\sigma_{\chi,q^b}$ is zero.
\end{theorem}
The main feature of Theorem~\ref{thm:cass_sig} is computability. Miller and Powell \cite{MP} give an algorithm to compute the twisted Blanchfield pairing
using Fox differential calculus. The methods of \cite{BCP}, which we presented in this article, allow us to compute the Casson-Gordon signatures.
As an application \cite{BCP} and later \cite{CKP} could prove non-sliceness of some linear combinations of iterated torus knots, generalizing 
previous results of Hedden, Kirk and Livingston \cite{HKL}.

\subsection{A closing remark}\label{sub:closing}
The two decomposition results: the classification of HVS of Theorem~\ref{thm:classification}
and the classification of real Blanchfield forms in Theorem~\ref{thm:class} share many properties.
There are some differences, which we want now to resume.

The classification of HVS deals much more efficiently with the case $\mu=1$, because
of the special definition of a simple HVS for $\mu=1$. The case of $(t-1)$-torsion modules
in the theory of linking forms causes notorious technical difficulties.

The classification of Blanchfield forms can be done in a more general setting, see for example \cite[Section 3]{Milnor_iso}.
Also, the notion of a Blanchfield form seems to be more universal and easier to adapt in different situations. For example, Seifert forms seem to be too rigid, that
an easy definition `Seifert forms twisted by a representation' seem possible.
\bibliographystyle{alpha}
\def\MR#1{}
\bibliography{research}

\end{document}